\documentclass[12pt,letterpaper]{amsart}
\usepackage{amsmath, amsthm, amssymb, mathrsfs}
\usepackage[tmargin=1.35in,bmargin=1.2in,rmargin=1.4in,lmargin=1.4in]{geometry}
\usepackage[breaklinks=true]{hyperref}
\usepackage{tikz-cd}

\theoremstyle{plain}
\newtheorem{theorem}{Theorem}[section]

\newtheorem{lemma}[theorem]{Lemma}
\newtheorem{prop}[theorem]{Proposition}

\newtheorem{question}{Question}

\theoremstyle{definition}

\newtheorem{remark}[theorem]{Remark}
\newtheorem*{acknowledgement}{Acknowledgements}
\newtheorem*{concludingremark}{Concluding remark}

\DeclareMathOperator{\spec}{\ensuremath{Spec}}
\DeclareMathOperator{\gr}{\ensuremath{gr}}
\DeclareMathOperator{\ad}{\ensuremath{ad}}
\DeclareMathOperator{\ch}{\ensuremath{char}}

\DeclareMathOperator{\Mod}{\ensuremath{-Mod}}
\DeclareMathOperator{\Hom}{\ensuremath{Hom}}
\DeclareMathOperator{\End}{\ensuremath{End}}
\DeclareMathOperator{\Pic}{\ensuremath{Pic}}
\DeclareMathOperator{\Aut}{\ensuremath{Aut}}

\begin{document}
\title[]{Morita equivalence of deformations of Kleinian singularities}
\author{Akaki Tikaradze}
\email{Akaki.Tikaradze@utoledo.edu}
\address{University of Toledo, Department of Mathematics \& Statistics, 
Toledo, OH 43606, USA}

\begin{abstract}
In this note we classify all Morita equivalent pairs of (classical)
generalized Weyl algebras for generic values of the parameters, thus
positively settling a 30 year old question posed by T.~Hodges. We also
prove a similar result for noncommutative deformations of arbitrary
Kleinian singularities provided the corresponding parameters are very
generic.
\end{abstract}

\maketitle

\section{Introduction}

Recall that classical generalized Weyl algebras introduced by V.~Bavula,
also considered by T.~Hodges under the name of noncommutative
deformations of type $A$ singularities \cite{H2}, are defined as follows.
Let $\bold{k}$ be a field and let $v\in \bold{k}[h]$.
Then the corresponding generalized Weyl algebra $A(v)$ is defined as a
quotient of $\bold{k}\langle x, y, h\rangle$
by the following relations
\[
[h, x]=x, \quad [h, y]=-y, \quad xy=v(h), \quad yx=v(h-1).
\]

Generalized Weyl algebras have been extensively studied in the
literature. The isomorphism problem for generalized Weyl algebras (over
$\mathbb{C}$) was solved by Bavula and Jordan \cite{BJ}. Namely, they
proved that $A(v)\cong A(v')$ if and only if $v'(h)=a v(bh+c)$ with
$a,b\in\mathbb{C}^*, c\in\mathbb{C}.$ It is not surprising that the
Morita equivalence problem turned out to be more difficult.

From now on, we say that $v\in\mathbb{C}[h]$ is generic if $v$ has no
multiple roots and the differences between its distinct roots are not
integers. It is well-known that the global dimension of $A(v)$ equals
1 precisely for generic values of $v.$

It was observed by Hodges \cite{H2} that (for a generic $v$) $A(v)$ is
Morita equivalent to $A(v')$ if the roots of $v$ and $v'$ differ by
integers. Earlier, it was also proved by Hodges \cite{H1} that
$U(\mathfrak{sl}_2)/(\Delta-\lambda),
U(\mathfrak{sl}_2)/(\Delta-\lambda')$ are Morita equivalent if and only
if $\lambda'\in \pm\lambda+\mathbb{Z}$ (here $\lambda\notin\mathbb{Z}$).
Recall that the algebras $U(\mathfrak{sl}_2)/(\Delta-\lambda)$ (maximal
primitive quotients of $U(\mathfrak{sl}_2)$) correspond to $A(v)$ with
quadratic $v.$
In other words, if $\deg(v)=2$, then $A(v), A(v')$ are Morita equivalent
only when they obviously are. This led Hodges \cite{H2} to pose the
following question.

\begin{question}[Hodges]
Let $A(v)$ and $A(v')$ be Morita equivalent, with $v$ generic. Then are
the algebras $A(v), A(v')$ obviously Morita equivalent?
\end{question}

The main goal of this paper is to give an affirmative answer to the above
question. Before stating it, recall that given an associative ring $R$,
its Picard group $\Pic(R)$ is defined as the group of isomorphism classes
of invertible $R$-bimodules. Given an automorphism $f\in \Aut(R)$ denote
by $R_{f}$ the $R$-bimodule which is $R$ as a left module, and on which
the right action is given by $f.$ Then $R_{f}\in\Pic(R)$ and this way we
may view $\Aut(R)$ as a subgroup of $\Pic(R).$ Our main result is as
follows.


\begin{theorem}\label{main}

Let $v=\prod_{i=1}^n(h-t_i)\in \mathbb{C}[h]$ be generic. 
Then $\Pic(A(v))=\Aut(A(v)).$ Let $B$ be a domain Morita equivalent to
$A(v).$ Then either\\ $\Pic(B)/\Aut(B)$ is uncountable, so in particular
$B$ is not isomorphic to any $A(v'),$ or
\[
B\cong A(v'), \qquad v'(h) = \prod_{i=1}^n(h-t_i+d_i)
\]
for some $d_i\in \mathbb{Z}.$

\end{theorem}

A similar result is also obtained in the upcoming paper \cite{CEEF2}.

A few comments about the above theorem are in order. The case of linear
$v$ corresponds to the first Weyl algebra $A(v)\cong A_1(\mathbb{C}),$ in
which case the above result is due to Stafford \cite{S}; in fact, he
showed that the result holds in arbitrary characteristic. In particular,
it was Stafford who emphasized the importance of the invariant
$\Pic(R)/\Aut(R)$ to distinguish between non-isomorphic Morita equivalent
domains. The case of quadratic $v$ was proved in \cite{CEEF1}.

Our proof crucially relies on two nontrivial results in the literature:
the first being the $\Aut(A(v))$-equivariant description of isomorphism
classes of ideals in $A(v)$ in terms of certain quiver varieties due to
F.~Eshmatov \cite{E} and Berest--Chalykh--Eshmatov \cite{BCE}; and the
second one being the transitivity of the action of $\Aut(A(v))$ on those
varieties proved in  \cite{CEET}.

Noncommutative deformations of all Kleinian singularities, i.e.\ the
algebras $\mathcal{O}^{\lambda}$ (which incorporate generalized Weyl
algebras $A(v)$), were introduced by Crawley-Boevey and Holland
\cite{CBH}.
We are able to show an analogous result for Morita equivalences between
these algebras provided that the parameter $\lambda$ is very generic (its
entries are algebraically independent over $\overline{\mathbb{Q}}$), see
Theorem \ref{kleinian}.

\section{Noncommutative deformations of general Kleinian singularities}

We start by briefly recalling the definition of deformed preprojective
algebras introduced by Crawley-Boevey and Holland \cite{CBH}. Let $Q$ be
an arbitrary finite quiver with the set of vertices $I$, and let
$\overline{Q}$ be its double. For each $i\in I$, denote by $e_i$ the
corresponding idempotent. Let $\lambda\in\mathbb{C}^I.$ Then the deformed
preprojective algebra $\Pi^{\lambda}(Q)$ with parameter $\lambda$ is
defined as the quotient of the path algebra of $\overline{Q}$ by the
following relations:
\[
\sum_{a\in Q}[a, a^*]=\sum_{i\in I} \lambda_ie_i.
\]

Next, we need to recall the reflection functors (for $\lambda_i\neq 0$)
\[
E_i:\Pi^{\lambda}(Q) \Mod \to \Pi^{r_i(\lambda)}(Q) \Mod
\]
satisfying $E_i^2={\rm Id}$, introduced by Crawley-Boevey and Holland
\cite{CBH}. First, they define the symmetrized Ringel form $(\cdot,\cdot)
: \mathbb{Z}^I \times \mathbb{Z}^I \to \mathbb{Z}$, via:
\[
(\alpha, \beta) = \langle \alpha, \beta \rangle + \langle \beta, \alpha
\rangle, \quad \text{where} \quad
\langle \alpha, \beta \rangle := \sum_{i \in I} \alpha_i \beta_i -
\sum_{a \in Q} \alpha_{t(a)} \beta_{h(a)}.
\]
Next, define a simple reflection for a given loop-free vertex $i\in I$ as
follows:
\[
s_i:\mathbb{Z}^{I}\to\mathbb{Z}^I, \quad s_i(\alpha)=\alpha-(\alpha,
\epsilon_i)\epsilon_i,
\]
where $\epsilon_i$ is the $i$-th coordinate vector, and the corresponding
reflection $r_i:\mathbb{C}^I\to \mathbb{C}^I$ defined with the property
that $r_i(\lambda)\cdot\alpha=\lambda\cdot s_i(\alpha),$ explicitly 
\[
r_i(\lambda)_j=\lambda_j-(\epsilon_i, \epsilon_j)\lambda_i.
\]
One now defines the Weyl group $W$ of $Q$ as the group of automorphisms of
$\mathbb{Z}^I$ generated by $s_i, i\in I.$
Note that $W$ acts on $\mathbb{C}^I$ by
$w(\lambda)\cdot\alpha=\lambda\cdot w^{-1}(\alpha)$
for all $w\in W, \lambda\in\mathbb{C}^I, \alpha\in \mathbb{Z}^{I}.$

Given a left $\Pi^{\lambda}$-module $M$, we put $M_i=e_iM$, where recall
that $e_i\in\Pi^{\lambda}$ denotes the idempotent of the vertex $i\in I.$
So, $M=\bigoplus_{i\in I}M_i.$ Then $(E_i(M))_j=M_j$ for $j\neq i.$ Also,
$(E_i(M))_i$ is a direct summand of $\bigoplus_{a\in Q, h(a)=i}
M_{t(a)}.$
By composing reflection functors, for any element of the Weyl group $w\in
W$, we get the functor, which is an equivalence of categories
\[
E_w:\Pi^{\lambda} \Mod \to \Pi^{w(\lambda)} \Mod.
\]

Given $\alpha\in\mathbb{N}^{I},$ denote by $M_{Q}(\lambda, \alpha)$ the
variety of isomorphism classes of $\Pi^{\lambda}(Q)$-modules of dimension
$\alpha.$ Then for each $w\in W,$ the translation functor $E_w$ induces
an isomorphism of algebraic varieties 
\[
R_w: M_Q(\lambda, \alpha)\cong M_Q(w(\lambda), w(\alpha)).
\]

Let $\lambda\in Z(\mathbb{C}[\Gamma]),$ where $\Gamma$ is a finite
subgroup of $SL_2(\mathbb{C}).$
Then to this datum Crawley-Boevey and Holland \cite{CBH} associated an
algebra $\mathcal{O}^{\lambda}$, which is a noncommutative deformation of
the Kleinian singularity $\mathbb{C}^2/\Gamma.$
So, $\mathcal{O}^{\lambda}$ is a filtered algebra such that its
associated graded algebra is isomorphic to
$\mathcal{O}(\mathbb{C}^2/\Gamma).$
Recall the definition of $\mathcal{O}^{\lambda}.$ At first, one defines
the corresponding symplectic reflection algebra  
\[
H^{\lambda}=(\Gamma\ltimes \mathbb{C}\langle x, y\rangle)/([x,
y]-\lambda).
\]
Then $\mathcal{O}^{\lambda}=eH^{\lambda}e$, where
$e=\frac{1}{|\Gamma|}\sum_{g\in\Gamma}g.$
Recall that the algebras $H^{\lambda}$ admit well-known higher
dimensional generalizations called symplectic reflection algebras
introduced by Etingof and Ginzburg \cite{EG}.

Now let us recall how the algebras $\mathcal{O}^{\lambda}$ relate to
deformed preprojective algebras. Let $Q$ be the McKay graph of $\Gamma$
with arbitrary orientation, thus vertices of $Q$ correspond to simple
representations $V_i$ of $\Gamma$ (the trivial module corresponds to
vertex 0), and with an arrow $i\to j$ if $\Hom_{\Gamma}(V\otimes V_i,
V_j)\neq 0.$ 
Let $I$ denote the set of vertices of $Q.$
Then, an important result of \cite{CBH} asserts that 
$\Pi^{\lambda}(Q)$ is Morita equivalent to $H^{\lambda}$, and
$\mathcal{O}^{\lambda}\cong e_0\Pi^{\lambda}e_0.$
Here we first write $\lambda\in\mathbb{C}^I$ as
$\lambda=\sum_{i \in I} f_i \tau_i$, where
$\tau_i$ is the idempotent in
$Z(\mathbb{C}[\Gamma])$ corresponding to the simple $\Gamma$-module
$V_i.$ Now $\lambda$ has $i$-th entry $f_i\delta_i$, where $\delta_i :=
\dim(V_i)$.
Moreover, the algebras $\Pi^{\lambda}$ and $\mathcal{O}^{\lambda}$ are
Morita equivalent if $\lambda\cdot \alpha\neq 0$ for all Dynkin roots
$\alpha.$

From now on, we say that $\lambda\in\mathbb{C}^{I}$ is generic if
$\lambda\cdot\alpha\neq 0$ for any root $\alpha.$ Similarly,
$\lambda\in\mathbb{C}^{I}$ is said to be very generic if its coordinates
are algebraically independent over $\mathbb{Q}.$ Recall that
$\mathcal{O}^{\lambda}$ is commutative if and only if
$\lambda\cdot\delta=0$, and  $\text{gl.dim}(\mathcal{O}^{\lambda})=1$ if
and only if $\lambda$ is generic \cite[Theorem 0.4]{CBH}.

Let $W_{\text{ext}}$ be the group generated by the reflections $s_i$ and
the diagram automorphisms of $Q$. So $W_{\text{ext}}$ is the extended
affine Weyl group of the affine root system of $Q.$ Hence given a generic
$\lambda \in \mathbb{C}^{I},$ for any $w\in W_{\text{ext}}$, we have the
corresponding equivalence of categories
\[
\bar{E}_{w}:\mathcal{O}^{\lambda} \Mod \to \mathcal{O}^{w(\lambda)} \Mod.
\]
It is natural to pose the analogue of Hodges' question in this more
general setup. We show that the answer is yes for a very generic
$\lambda$ (defined just above). Here we also need to recall that as
observed by Boyarchenko \cite{B}, translations by 
\begin{equation}\label{EMitya}
d\in\Lambda := \lbrace \xi\in \mathbb{Z}^I,\
\xi \cdot \delta = 0 \rbrace
\end{equation}
provide Morita equivalent pairs of algebras $\mathcal{O}^{\lambda}$ for
generic $\lambda$ such that $\lambda \cdot \delta=1$:
\[
\mathcal{O}^{\lambda} \Mod \cong \mathcal{O}^{\lambda+d} \Mod.
\]

We need the following result which should be well-known. Its proof easily
follows along the lines of considering the standard Koszul resolution for
symplectic reflection algebras, see \cite[proof of Theorem 1.8]{EG}.

\begin{lemma}\label{H0}
Let $\lambda\in Z(\mathbb{C}[\Gamma])$ be very generic.
Then the natural map
\[
Z(\mathbb{C}[\Gamma])\to H^{\lambda}/[H^{\lambda}, H^{\lambda}]
\]
is surjective with kernel $\mathbb{C}\lambda.$
\end{lemma}

The next result is a direct corollary of using the Hattori--Stallings
traces just as was originally done by Hodges in \cite{H2}.
This technique was also used by Berest, Etingof, and Ginzburg to classify
Morita equivalent pairs of rational Cherednik algebras of $S_n$ for
generic parameters \cite{BEG}.
Our result is a generalization of a result by Richard--Solotar
\cite[Theorem 2.6.2]{RS} to noncommutative deformations of arbitrary
Kleinian singularities.

Given $\lambda\in\mathbb{C}^{I}$, define $\bar{\lambda}:=(\lambda_i,
i\neq 0)\in\mathbb{C}^{|I|-1}.$
Recall that given a finitely generated projective $A$-module $P$, its
Hattori--Stallings trace is defined as the image in $A/[A, A]=H_0(A)$ of
the trace of an idempotent corresponding to $P.$ Thus we have the trace
homomorphism
\[
tr = tr_A : K_0(A) \to H_0(A).
\]

\begin{prop}\label{traces}
Let $H^{\lambda}, H^{\lambda'}$ be Morita equivalent, with $\lambda \cdot
\delta = \lambda' \cdot \delta = 1.$
Then there exist $A\in GL_{n-1}(\mathbb{Z})$ and $d\in\mathbb{Z}^{n-1}$
-- where $n = |I|$ -- such that $\bar{\lambda'}=A\bar{\lambda}+d.$
\end{prop}

\begin{proof}
Recall that $\tau_i\in Z(\mathbb{C}[\Gamma])$ denotes the primitive
central idempotent corresponding to $V_i,$ and
$\lambda=\sum_{i \in I} f_i\tau_i,
\lambda'=\sum_{i \in I} f'_i\tau_i.$ Moreover, the vectors
$\lambda, \lambda'\in\mathbb{C}^I$ have $i$-th coordinate $f_i \delta_i,
f'_i \delta_i$ respectively. So by our assumption $\sum_i f_i \delta_i^2
= \sum_i f'_i \delta_i^2 = 1.$

Let $F:H^{\lambda} \Mod \to H^{\lambda'} \Mod$ be an equivalence of
categories, then $F$ preserves the Hattori--Stallings traces. Also, as
$F$ must descend to an equivalence $\mathcal{O}^{\lambda} \Mod \to
\mathcal{O}^{\lambda'} \Mod$, so $F$ must preserve the rank, where we
define:
$rank_{\lambda}(M) := rank_{\mathcal{O}^{\lambda}}(eM).$
It is well-known that $K_0(H^{\lambda})=K_0(\mathbb{C}[\Gamma])$ has a
basis $[P_i], i\in I$, where $P_i=H^\lambda
\otimes_{\mathbb{C}[\Gamma]}V_i$ \cite{BEG}.
Also, $rank([P_i])=\delta_i$ and $tr([P_i])=\tau_i/\delta_i.$ 

Denote by $\tilde{K}(H^{\lambda})$  the kernel of the rank homomorphism
$rank:K_0(H^{\lambda})\to\mathbb{Z}.$
Then $\lbrace [P_i]-\delta_i[P_0] =: \varepsilon_i, i>0\rbrace $ is a
$\mathbb{Z}$-basis of $\tilde{K}(H^\lambda).$ 
Since $\sum_{i \in I} f_i \tau_i=0$ in $H^{\lambda}/[H^{\lambda},
H^{\lambda}]$ and
\[
\tau_i=tr(\delta_i \varepsilon_i)+\delta_i^2\tau_0, \quad i>0,
\]
we conclude that (setting $tr_\lambda := tr_{H^\lambda}$ and
$tr_{\lambda'} := tr_{H^{\lambda'}}$)
\[
tr_{\lambda}^{-1}(e_0)=-\sum_{i>0}\delta_i f_i \varepsilon_i,\quad
tr_{\lambda'}^{-1}(e_0)=-\sum_{i>0}\delta_i f'_i \varepsilon_i.
\]
Since $rank(F(H^{\lambda}e_0))=1$,  the desired result follows from the
following commutative diagram
\center
\begin{tikzcd}[swap]\label{diagram}
    \tilde{K}_0(H^{\lambda}) \arrow{r}{K_0(F)} 
           \arrow{d}{\text{tr}_{\lambda}}
  & \tilde{K}_0(H^{\lambda'})  \arrow{d}[swap]{\text{tr}_{\lambda'}} \\  
    H_0(H^{\lambda}) \arrow{r}[swap]{H_0(F)} 
  & H_0(H^{\lambda'})                                     
\end{tikzcd}

\end{proof}

We next recall the following simple result. It should be well-known, but
we give a proof here for the sake of completeness.

\begin{lemma}\label{roots}
Let $R$ be an affine root system spanning an $\mathbb{R}$-vector space
$V.$ Let $T\in GL(V)$ such that there is a set $S\subset R$ containing
all imaginary roots and finitely many other roots, such that
$T(R\setminus S), T^{-1}(R\setminus S)\subset R.$ Then $T$ is an element
of the extended affine Weyl group of $R.$
\end{lemma}

\begin{proof}
Let $\delta$ denote the minimal imaginary positive root of $R$. Let $R'$
denote the finite root system whose affinization is $R.$
So, $R=(R'\times\mathbb{Z}\delta) \sqcup \mathbb{Z} \delta.$ 
Thus for all but finitely many $\alpha\in R\setminus\mathbb{Z}\delta$ we
have $T(\alpha)\in R.$
Recall that for given roots $\alpha,\beta,$ if $m\alpha+\beta\in R$ for
infinitely many $m$, then $\alpha$ must be a multiple of $\delta.$
This implies $T(\delta)=\pm\delta.$ Thus, without loss of generality
$T(\delta)=\delta.$ Then $\bar{T}:V/\mathbb{R}\delta\to
V/\mathbb{R}\delta$ is a linear transformation preserving the finite root
system $\bar{R} \cong R'.$ Now, it easily follows that $\bar{T}$ belongs
to the subgroup generated by the Weyl group and diagram automorphisms of
$\bar{R}.$
Hence $T$ belongs to the extended affine Weyl group of $R.$
\end{proof}

\begin{theorem}\label{kleinian}

Let $\mathcal{O}^{\lambda}$ and $\mathcal{O}^{\lambda'}$ be Morita
equivalent noncommutative deformations of a Kleinian singularity. If
$\lambda$ is very generic then $\lambda'=tw(\lambda)$ for some $w\in
W_{\text{ext}}$ and $t\in\mathbb{C}^*.$

\end{theorem}

\begin{proof}

Without loss of generality we may assume that
$\lambda\cdot\delta=\lambda'\cdot\delta=1.$
By Proposition \ref{traces}, we have that $\lambda'=A\lambda+d$ for some
$A\in GL_{n}(\mathbb{Z})$ and $d\in\Lambda$ (see~\eqref{EMitya}), where
$n=|I|.$
Now, recall that translations by $d\in\Lambda$ induce Morita equivalences
and belong to $W_{\text{ext}}.$ Thus, without loss of generality
we may assume that  $\lambda'=A\lambda.$
It follows that there exists a finitely generated ring
$S\subset\mathbb{C}$, containing $\lambda_0,\cdots, \lambda_{n-1},$ such
that $\mathcal{O}^{\lambda}, \mathcal{O}^{\lambda'}$
are Morita equivalent over $S.$
Denote by $T:\mathbb{R}^n\to\mathbb{R}^n$ the linear transformation
$T(v)=Av.$
By Lemma \ref{roots}, we may assume without loss of generality that for
infinitely many non-imaginary roots $\alpha$, $T(\alpha)$ is not a root.
It follows that there exists a homomorphism $\chi:S\to\mathbb{\bar{Q}}$
and a non-imaginary root $\alpha,$ so that
$\chi(\lambda)\cdot \alpha=0$, while $\chi(\lambda')\cdot \beta\neq 0$
for any root $\beta.$ So,
$\mathcal{O}^{\chi(\lambda)},\mathcal{O}^{\chi(\lambda')}$
are (still) Morita equivalent, while
$\text{gl.dim}(\mathcal{O}^{\chi(\lambda)})\neq
\text{gl.dim}(\mathcal{O}^{\chi(\lambda')}),$ 
a contradiction.
\end{proof}

\section{Characteristic $p\gg 0$ approach}

The goal of this section is to prove Theorem \ref{normality} below, which
plays a key role in proving our main result.
Its proof is based on the reduction mod $p\gg 0$ method. To state the
result, we need to introduce the following subgroup of $\Aut(A(v)).$
Let $\bold{k}$ be a field, $v\in\bold{k}[h].$
Let $G$ be the subgroup of $\Aut(A(v))$ generated by 
$\exp(\lambda \ad x^n), \exp(\lambda \ad y^n)$ for $\lambda\in \bold{k},
n\geq 1,$ 
and the diagonal automorphisms 
$\phi_t, t\in\bold{k}^*,$ defined as 
\[
\phi_t(x)=tx, \quad \phi_t(y)=t^{-1}y, \quad \phi_t(h)=h.
\]
Notice that the definition of $G$ makes sense when
$\ch \bold{k}>0$ also.
Following \cite{BJ}, we call a polynomial $v\in\bold{k}[h]$ reflexive if
$v(h)=(-1)^dv(a-h)$ for some $a\in\bold{k}, d=\deg(h).$
If $v$ is reflexive, then we have an involution $\Omega\in \Aut(A(v))$
defined as follows:
\[
\Omega(x)=y,\quad \Omega(y)=(-1)^dx,\quad \Omega(h)=1+a-h.
\]
It was proved by Bavula and Jordan \cite[Theorem 3.29]{BJ} that if $v$
is reflexive, then $\Aut(A(v))$ is generated by $G$ and $\Omega$,
otherwise $\Aut(A(v))=G.$ Recall that $\Omega$ normalizes $\Gamma.$

For $v(h) = \prod_i (h-c_i)\in \mathbb{C}[h]$, let $S\subset\mathbb{C}$
be a finitely generated subring of $\mathbb{C}$ containing all roots of
$v.$
Given a base change $S\to\bold{k}$ to a characteristic $p$ field
$\bold{k},$ we set $v^{[p]}(h) := \prod_i(h-(c_i^p-c_i))\in
\bold{k}[h].$

We also need to recall the commutative deformations of type $A$ Kleinian
singularities: the commutative Poisson $\bold{k}$-algebra $B(v)$
generated by $x, y, h$ with the relation $xy=v(h)$ and the following
Poisson brackets 
\[
\lbrace h, x\rbrace=x, \quad \lbrace h, y\rbrace=-y, \quad \lbrace x,
y\rbrace=v'(h).
\]

We need the following characteristic $p\gg 0$ analogue of the Poisson
version of the aforementioned theorem of Bavula--Jordan on automorphisms
of $A(v).$
(The original proof in \cite{BJ} also goes through directly for this
analogue.)
If $v$ is reflexive, then $B(v)$ admits the Poisson automorphism $\Omega$
defined just as for $A(v).$
The subgroup $W\leq \Aut(B(v))$ of Poisson automorphisms is defined
similarly to the subgroup $G\leq\Aut(A(v)).$ Now one has:

\begin{theorem}\label{p-BJ}
Given a natural number $d$ there exists $N\in \mathbb{N},$ such that for
all primes $p>N,$ the following holds.
Let $\bold{k}$ be an algebraically closed field of characteristic $p$,
and let $v\in\bold{k}[h].$
Let $\phi:B(v)\to B(v)$ be a Poisson $\bold{k}$-algebra automorphism,
such that $\deg(\phi(x)), \deg(\phi(y)), \deg(v)\leq d.$
Then $\phi$ belongs to the subgroup generated by $\Omega$ and $W$ if $v$
is reflexive, otherwise $\phi\in G.$
\end{theorem}

Recall that the center of $A(v)$ is trivial for $\ch \bold{k} =0$,
while in the case of $\ch \bold{k}=p$ it is large and is described as
follows \cite{BC}. The center  $Z(A(v))$ is generated by $x^p, y^p,
h^p-h$ subject to the relation
\[
x^py^p=\prod_i(h^p-h-(c_i^p-c_i)),
\]
where $v(h)=\prod_i(h-c_i).$

Recall that given an associative flat $\mathbb{Z}$-algebra $R$ and a
prime number $p,$ the center $Z(R/pR)$ of its reduction mod $p$ acquires
a natural Poisson bracket (see for example \cite[Section 5.2]{BK}), which
we refer to as the reduction mod $p$ bracket, defined as follows. Given
$a, b\in Z(R/pR)$, let $z, w\in R$ be their respective lifts. 
Then the Poisson bracket $\lbrace a, b \rbrace$ is defined to be
\[
\frac{1}{p}[z, w] \mod p\in Z(R/pR).
\]

So, given a base change $S\to\bold{k}$ to a characteristic $p$ field
$\bold{k}$, where $S\subset\mathbb{C}$ is a finitely generated ring
containing all roots of $v,$ then $Z(A(v)\otimes_S\bold{k})$ is equipped
with the reduction mod $p$ Poisson bracket as defined above and we get an
isomorphism
\[
Z(A(v)\otimes_S\bold{k})\cong B(v^{[p]}),
\]
preserving the Poisson brackets with the minus sign (see \cite[Section
5.2]{BK} for the case of the Weyl algebra; the general case is similar).

\begin{remark}
It follows from the realization of $\mathcal{O}^{\lambda}$ (for
$\lambda\cdot\delta=1)$ as a quantum Hamiltonian reduction \cite{Hol}
that in characteristic $p\gg 0$ we have that 
$$Z(\mathcal{O}^{\lambda}\mod p)\cong \mathcal{O}^{\lambda^{[p]}},$$
 where
$\lambda^{[p]}_i=\lambda_i^p-\lambda_i.$ So,
$\lambda^{[p]}\cdot\delta=0.$ Recall that $\mathcal{O}^{\lambda}$ is
commutative if and only if $\lambda\cdot\delta=0.$ In particular, if
$\mathcal{O}^{\lambda_1}, \mathcal{O}^{\lambda_2}$ are Morita equivalent,
then we have an isomorphism of Poisson algebras
$\mathcal{O}^{\lambda_1^{[p]}}\cong \mathcal{O}^{\lambda_2^{[p]}}$ for
$p\gg 0.$
\end{remark}

\begin{theorem}\label{normality}
Let $v\in\mathbb{C}[h]$ be arbitrary.
Then $\Aut(A(v))$ is a normal subgroup of $\Pic(A(v)).$
Moreover, if $v$ is reflexive, or the roots of $v$ are algebraically
independent over $\overline{\mathbb{Q}}$, then $\Pic(A(v))=\Aut(A(v)).$

\end{theorem}

\begin{proof}
Recall that given an automorphism $f$ of an algebra $A$, by $(A)_f$ we
denote the $A$-bimodule which is $A$ as a left module and on which the
right action is by $f.$
Let $M\in \Pic(A(v))$. The idea is to show that $M\Aut(A(v)) M^{-1}\mod
p\subset \Aut(A(v))$ for all $p\gg 0.$ 
Recall that given a bimodule $N$ over $A(\bar{v}),$ where
$\bar{v}\in\bold{k}[h]$ and $\bold{k}$ is a characteristic $p$ field, the
support of $N$ (denoted by $\text{Supp}(N)$) is the closed subvariety of
$\spec(Z(A(\bar{v}))\otimes Z(A(\bar{v})))$ corresponding to the
annihilator of $N$ in $Z(A(\bar{v}))\otimes Z(A(\bar{v})).$
A standard argument shows that for $N\in \Pic(A(\bar{v})),$ its support
must be the graph of an automorphism $\phi\in \Aut(Z(A(v))),$ and $N$ is
uniquely determined by its support.


At first, we consider the case when $v$ is not reflexive, so
$\Aut(A(v))=G.$ 
Let $N'$ be a bimodule corresponding to an element $f\in G,$ so
$N'=(A(v))_f$. Let $M\in\Pic(A(v))$ and put $N=MN'M^{-1}\in \Pic(A(v)).$
We need to show that $N\cong A$ as a left module.
Let $S\subset\mathbb{C}$ be a large enough finitely generated ring
containing the roots of $v,$ such that all the above modules are defined
over it.
We first verify that for any base change $S\to\bold{k}$ to an
algebraically closed field of characteristic $p,$ we have 
\[
N_{\bold{k}}=N\otimes_S\bold{k}\cong A(\bar{v}),
\]
where $\bar{v}$ denotes the image of $v$ in $\bold{k}[h].$
Indeed, let $\phi\in \Aut(Z(A(\bar{v})))$ be the automorphism whose graph
is the support of $M_{\bold{k}}.$ Recall that $Z(A(\bar{v})))\cong
B(v^{[p]}).$
So, we may view $\phi$ as a Poisson automorphism of $B(v^{[p]}).$
We also see that $f|_{Z(A(\bar{v}))}=\bar{f}\in W.$ By using Theorem
\ref{p-BJ}, we conclude that $\phi$ belongs to the subgroup generated by
$\Omega$ and $W$ (and $\phi\in W$ if $v^{[p]}$ is not reflexive).
Since $\Omega$ normalizes $W,$ it follows that $\phi\bar{f}\phi^{-1}\in
W.$ So, $\text{Supp}(N_{\bold{k}})$ is the graph of an element of $W.$
Since $G|_{Z(A(\bar{v}))}=W$, we conclude that there exists $\psi\in G$
so that $N_{\bold{k}}\cong A(\bar{v})_{\psi}.$
Thus, we see that $N_{\bold{k}}\cong A(\bar{v})$ as left modules. 

Since $N$ as a left module has rank 1, 
we may assume that $N$ is a left ideal in $A(v).$
Moreover, we may assume that $N$ is a submodule of $A(v)$ over $S.$
Let us  equip $N$ with the induced filtration from $A(v).$
Then $N$ is a principal ideal if and only if $\gr(N)$ is a principal
ideal in $\gr(A(v)).$ Now, since we have proved that $N_{\bold{k}}\cong
A(\bar{v})$, it follows that $\gr(N)_{\bold{k}}$ is a principal ideal in
$\gr(A(\bar{v}))$. Hence, $\gr(N)$ is a principal ideal and we are done.

Now suppose $v$ is reflexive, then so is $v^{[p]}.$ Then any Poisson
automorphism of $B(v^{[p]})=Z(A(\bar{v}))$ that belongs to the subgroup
generated by $\Omega$ and $W$ can be lifted to an automorphism of
$A(\bar{v})$, and the rest of the proof is identical to the case of a
non-reflexive $v.$

Finally, suppose $v$ satisfies one of the assumptions in the lemma. Then
it follows from the above considerations that $M_{\bold{k}}\in
\Aut(A(\bar{v}))$ for infinitely many $\bold{k}$ with $\ch \bold{k} \gg
0$. Then as explained above, it follows that $M$ belongs to the image of
$\Aut(A(v))$ in the Picard group.
\end{proof}

\begin{remark}
The usage of the reduction mod $p\gg 0$ technique for statements of the
form $\Pic(R)=\Aut(R)$, where $R$ is a filtered quantization of a Poisson
algebra is quite typical: it was used by Stafford for the Weyl algebra
\cite{S}. It is also worth mentioning that a recent work of C.~Dodd
\cite{D} showed that $\Pic(A_n(\mathbb{C}))=\Aut(A_n(\mathbb{C}))$,
settling a conjecture by Belov-Kanel and Kontsevich, where the reduction
mod $p\gg 0$ approach is at the heart of everything.
\end{remark}

\section{Classification of ideals via quiver varieties}

Let $\mathcal{O}^{\lambda}$ be a noncommutative deformation of a Kleinian
singularity, with $\lambda$ generic (so
$\text{gl.dim}(\mathcal{O}^{\lambda})=1$).
Then it was proved by Baranovsky, Ginzburg, and Kuznetsov \cite{BGK} that
there is a bijection between the set of isomorphism classes of
indecomposable projective $\mathcal{O}^{\lambda}$-modules and a countable
disjoint union of certain quiver varieties. 

In the case of deformations of type $A$ Kleinian singularities $A(v)$,
automorphism groups $\Aut(A(v))$ are very big, so it is interesting to
consider the natural action of $\Aut(A(v))$ on the set of isomorphism
classes of indecomposable projective modules. An explicit and
$\Aut(A(v))$-equivariant version of the aforementioned theorem of
Baranovsky--Ginzburg--Kuznetsov was proved by Eshmatov \cite{E} and
Berest, Chalykh, and Eshmatov \cite{BCE} (building on earlier works on
ideals of the Weyl algebra and the Calogero--Moser spaces by Berest,
Chalykh, and Wilson).
In this section we recall the bijection constructed in \cite{BCE} and
observe that it is compatible with reflection functors (Lemma
\ref{commute}), which plays a crucial role in our proof on Theorem
\ref{main}.

From now on, let $Q$ denote the quiver corresponding to the extended
Dynkin diagram $\tilde{A}_m$, which is a cycle with $m$ vertices. 
(This corresponds to the cyclic group $\Gamma$ of size $m$.)
Let $Q_{\infty}$ denote the quiver obtained from $Q$ by attaching an
additional vertex, labeled $\infty$, with an arrow to the vertex $0.$
Let $\lambda=(\lambda_0,\cdots, \lambda_{m-1})\in \mathbb{C}^m$ be a
regular (generic) vector. Also let $\alpha \in\mathbb{Z}^{m+1}$ be a
positive root for $Q_{\infty}$ with $\alpha_\infty = 1$, and write
$\alpha = (1, \eta)$. Set $\lambda_{\alpha} := (-\lambda\cdot \eta,
\lambda) \in \mathbb{C}^{m+1}$ and $B :=
e_{\infty}\Pi^{\lambda_{\alpha}}({Q}_{\infty})e_{\infty}.$ It is known
that
$\Pi^{\lambda}(Q)=\Pi^{\lambda_{\alpha}}(Q_{\infty})/(e_{\infty})$
\cite[Lemma 6]{BCE}.

It is well-known that a generalized Weyl algebra $A(v)$ can be identified
with $\mathcal{O}^{\lambda}=e_0\Pi^{\lambda}(Q)e_0$ for an appropriate
$\lambda.$
Indeed, let $\lambda=(\lambda_1,\cdots,
\lambda_{n-1})\in\mathbb{C}^{n-1},$ then define $a=(a_1,\cdots, a_n)$ as
follows: $a_n = 0$, and
\[
a_i=(n-i+\lambda_{i})/n, \quad i<n.
\]
Put $v=\prod_{i=1}^n(h-a_i).$ Then $A(v)\cong \mathcal{O}^{\lambda}$
\cite[Lemma 7.1]{M}.

In the above setting, Berest--Chalykh--Eshmatov considered the functor 
\[
F:\Pi^{\lambda_{\alpha}}(Q_{\infty}) \Mod \to \Pi^{\lambda}(Q) \Mod
\]
defined as follows
\[
F(M) := \Hom_{\Pi^{\lambda_{\alpha}}(Q_{\infty})}
(\Pi^{\lambda_{\alpha}}(Q_{\infty})/(e_{\infty}),
\Pi^{\lambda_{\alpha}}(Q_{\infty})\otimes_{B}Me_{\infty}).
\]

Recall that $M_{Q_{\infty}}(\lambda_{\alpha}, {\alpha})$ denotes the
quiver variety of isomorphism classes of
$\Pi^{\lambda_{\alpha}}(Q_{\infty})$-modules with dimension vector
$\alpha.$ From now on we denote $M_{Q_{\infty}}(\lambda_{\alpha}, {\alpha})$ by
$M_{Q_{\infty}}(\lambda, {\alpha})$ to simplify the notation.
Then as was proved in \cite{BCE}, the functor $F$ restricts on
$M_{Q_{\infty}}(\lambda_{\alpha}, {\alpha})$ to an injective function
with values in indecomposable projective modules over $\Pi^{\lambda}(Q).$
Thus, we have the following inclusion 

$$
F_{\lambda}:\bigsqcup_{\alpha} M_{Q_{\infty}}(\lambda, \alpha)\hookrightarrow
\lbrace  \text{isoclasses of indec.\ projective }
\Pi^{\lambda}(Q)\text{-modules}\rbrace.
$$
In order to extend $F_{\lambda}$ to a bijection, we need to consider for each vertex $i$ of $Q$ an analog of $Q_{\infty}$ defined  by adding to $Q$ 
the vertex $\infty$ with an arrow to the vertex $i.$
Specifically, given $i\in\mathbb{Z}/m\mathbb{Z}$, let $\lambda^i\in\mathbb{C}^m$ be obtained from $\lambda$ by rotating its coordinates by $i,$ 
so $(\lambda^i)_j=\lambda_{i+j}, j\in\mathbb{Z}/m\mathbb{Z}.$
Since cyclic rotations belong to diagram automorphisms of $Q$, we get that $\Pi^{\lambda^i}(Q)\cong\Pi^{\lambda}(Q),$ which induces the bijection
$$ \lbrace  \text{isoclasses of indec.\ proj. }
\Pi^{\lambda^i}(Q)\text{-mod}\rbrace\cong \lbrace  \text{isoclasses of indec.\ proj. }
\Pi^{\lambda}(Q)\text{-mod}\rbrace.$$
Thus, as was shown in \cite[Theorem 5]{E}, by combining the above bijection with functions $F_{\lambda^i}$ above,
we obtain the following bijection onto the set of isomorphism classes of indecomposable projective
$\Pi^{\lambda_{\alpha}}(Q)$-modules
$$
F_{\lambda}:\bigsqcup_{i\in\mathbb{Z}/m\mathbb{Z}}\bigsqcup_{\alpha} M_{Q_{\infty}}(\lambda^i, \alpha)\cong
\lbrace  \text{isoclasses of indec.\ projective }
\Pi^{\lambda}(Q)\text{-modules}\rbrace.
$$ 

Next, we need to check that $F$ commutes with the reflection functors.
This is the subject of the following lemma.

\begin{lemma}\label{commute}
The functor $F$ commutes with the reflection functors $E_{i}$ for all
$i\neq\infty.$
\end{lemma}

\begin{proof}
In what follows we denote $\Pi^{\lambda}(Q_{\infty})$ by $\Pi$, and
$\Pi^{r_i(\lambda)}(Q_{\infty})$ by $\Pi'.$ 
It suffices to show that the reflection functors commute with the
following functors:
\[
\Phi(-)=\Hom_{\Pi}(\Pi/\Pi e_{\infty}\Pi, -), \quad L: M\to
\Pi\otimes_{B}Me_{\infty}.
\]
For a $\Pi$-module $M$, put $M_i=e_iM$, so $M=\bigoplus_iM_i.$ Recall
that $(E_i(M))_j=M_j$ for $j\neq i$, while $(E_i(M))_i$ is a subspace of
$\bigoplus_{i\to j} M_j.$
We first verify that $E_i(\Phi(M))\subset \Phi(E_i(M)).$
Indeed, $E_i(\Phi(M))_j$ is the annihilator of $e_{\infty}\Pi e_j$ in
$M_j.$ Since $e_{\infty}\Pi'e_j$ belongs to the image of $e_{\infty}\Pi
e_j$ in $\Hom_{\mathbb{C}}(M_j, M_{\infty}),$ we get that
\[
E_i(\Phi(M))_j\subset \Phi(E_i(M))_j, j\neq i.
\]
Let
\[
v\in E_i(\Phi(M))_i\subset\bigoplus_{i\to j} Ann(e_\infty \Pi e_j).
\]
Write $v=\sum_{i\to j}v_j.$
It suffices to verify that $(e_{\infty}\Pi'e_i)v=0.$
This follows from the fact that any action by a path in $\Pi'$ ending in
the vertex $\infty$ can be written as sum of  actions by an element of
$e_{\infty}\Pi$, so
\[
(e_{\infty}\Pi'e_i)v\subset \sum_{i \to j} (e_{\infty}\Pi e_j)v_j=0.
\]
This concludes the proof of the inclusion $E_i(\Phi(M))\subset \Phi(E_i(M)).$
So, replacing $M$ by $E_i(M)$, we get that
\[
E_i(\Phi(E_i(M)))\subset \Phi(M).
\]
Now applying $E_i$ to both sides yields $\Phi(E_i(M))\subset
E_i(\Phi(M)).$ Thus, $\Phi E_i = E_i \Phi$ as desired.

Next, we verify that $E_iL$ and $LE_i$ represent the same functor, $i\neq
\infty.$ Indeed, 
\[
\Hom_{\Pi'}(E_i(L(M)), E_i(N)) = \Hom_{\Pi}(LM, N) = \Hom_{e_{\infty}\Pi
e_{\infty}}(e_{\infty}M, e_{\infty}N).
\]
On the other hand (recalling that $e_{\infty}\Pi
e_{\infty}=e_{\infty}\Pi'e_{\infty}$),
\[
\Hom_{\Pi'}(L(E_i(M)),
E_i(N))=\Hom_{e_{\infty}\Pi'e_{\infty}}(e_{\infty}M,
e_{\infty}E_i(N))=\Hom_{e_{\infty}\Pi e_{\infty}}(e_{\infty}M,
e_{\infty}N).
\]
So, $LE_i=E_iL$ and we are done.
\end{proof}

In what follows, we denote by $W_{\infty}$ the subgroup of the Weyl group
of $Q_{\infty}$ generated by the reflections $s_i,$ $i\neq\infty.$
Thus, by combining the bijection constructed by Berest--Chalykh--Eshmatov
\cite{BCE} given above with reflection functors, we obtain that for any
$w\in W_{\infty},$ the following diagram commutes:

\begin{tikzcd}[swap]
    \bigsqcup_{\alpha} M_{Q_{\infty}}(\lambda_{\alpha}, \alpha)
    \arrow{r}{F_{\lambda}} \arrow{d}{R_w}
  &\lbrace  \text{isoclasses of indec. projective }
  \mathcal{O}^{\lambda}\text{-modules}\rbrace  \arrow{d}[swap]{R_w} \\  
    \bigsqcup_{\alpha} M_{Q_{\infty}}(w(\lambda_{\alpha}), w(\alpha))
    \arrow{r}[swap]{F_{w(\lambda)}} 
  & \lbrace  \text{isoclasses of indec. projective }
  \mathcal{O}^{w(\lambda)}\text{-modules}\rbrace
\end{tikzcd}

\section{Proof of the main theorem}

Throughout this section we are assuming that the parameter
$v\in\mathbb{C}[h]$ is generic. In order to prove Theorem \ref{main} we
need the following result, which should be very standard. We include its
proof for the reader's convenience.

\begin{lemma}\label{orbit}
Let $A$ be a Noetherian domain. Let $P$ be an indecomposable projective
generator of $A \Mod$. Let $B=\End(P)^{op}.$ Let $X$ denote the set of
isomorphism classes of indecomposable projective $A$-modules. Then the
orbit of $P$ in $X$ under $\Pic(A)$ equals $\Pic(B)/\Aut(B).$ 
\end{lemma}

\begin{proof}
It suffices to observe that the stabilizer of $P$ in $\Pic(A)=\Pic(B)$ is
$\Aut(B).$
To show this, without loss of generality we may take $P=A.$ Then an
invertible $A$-bimodule $L\in \Pic(A)$ stabilizes $A$ if
$L\otimes_{A}A\cong A$ as left $A$-modules. So, $L=A$ as a left
$A$-module, and the invertibility of $L$ implies that the right action of
$A$ on $A$ must be given by an element of $\Aut(A).$ So, the stabilizer
of $A$ in $\Pic(A)$ equals the image of $\Aut(A)$ in $\Pic(A).$
\end{proof}

We also need the following result. Of course, eventually we show a much
stronger statement.

\begin{prop}\label{countable}
$\Pic(A(v))/\Aut(A(v))$ is at most countable.
\end{prop}

\begin{proof}
Let $X$ be the set of isomorphism classes of indecomposable projective
$A(v)$-modules. 
We know that the stabilizer of $A\in X$ is $\Aut(A)$.
We also know that there is a $G$-equivariant bijection between $X$ and
the disjoint union of countably many quiver varieties, such that $G$ acts
transitively on each of those quiver varieties \cite[Theorem 1.1]{CEET}.
Therefore, $X/G$ is countable. Since the orbit of $A$ in $X$ is
$\Pic(A)/G$, we conclude that $\Aut(A(v))\backslash \Pic(A)/\Aut(A(v))$
is at most countable. Now, since $\Aut(A(v))$ is a normal subgroup in
$\Pic(A(v))$ by Theorem \ref{normality}, we obtain that
$\Pic(A(v))/\Aut(A(v))$ is at most countable. 
\end{proof}

Finally, we show the main result of the paper. Recall that $W_{\text{ext}}$ denotes the extended
affine Weyl group of $Q$, which is generated by simple reflections and diagram automorphisms of $Q.$

\begin{proof}[Proof of Theorem \ref{main}]
Since we are using reflection functors, it is more convenient to work
with $\mathcal{O}^{\lambda}\cong A(v).$
Thus, the parameter $\lambda$ is generic and $\mathcal{O}^{\lambda'}$ is
Morita equivalent to $\mathcal{O}^{\lambda'}.$
We need to show that $\mathcal{O}^{\lambda'}\cong
\mathcal{O}^{w(\lambda)}$ for some $w\in W_{\text{ext}}.$ As before, $X$
denotes the set of isomorphism classes of indecomposable projective
$\mathcal{O}^{\lambda}$-modules. 
Given such a module $P$, we denote by $[P]\in X$ its isomorphism class.
Recall that $\Pic(\mathcal{O}^{\lambda})$ acts on $X.$ 

Now, let $P$ be an indecomposable projective
$\mathcal{O}^{\lambda}$-module with the property that the orbit $G \cdot
[P]$ is uncountable. Put $B=\End_{\mathcal{O}^{\lambda}}(P)^{op}.$
Then $B$ is Morita equivalent to $\mathcal{O}^{\lambda}$, so we may
identify $X$ with the set of isomorphism classes of indecomposable
projective $B$-modules.
Now the $\Pic(B)$-orbit of $[P]$ is uncountable, hence by Lemma
$\ref{orbit}$ $\Pic(B)/\Aut(B)$ is uncountable.
In particular, $B$ cannot be isomorphic to any $\mathcal{O}^{\mu}$ by
Proposition \ref{countable}.

Thus, using the bijection of Berest, Chalykh, and Eshmatov, we conclude
that $\mathcal{O}^{\lambda'}\cong \End_{\mathcal{O}^{\lambda}}(P)^{op},$
where $[P]\in F_{\lambda}(M_{Q_{\infty}}(\mu,\alpha))$ such that $\mu$ is obtained from a cyclic permutation of $\lambda$, 
and $\dim M_{Q_{\infty}}(\mu, \alpha) =0.$ Since $\mu=w(\lambda)$, where w is an element of the extended affine Weyl group of $Q$, we may assume without
loss of generality that $\lambda=\mu.$

Recall that $M_{Q_{\infty}}(\lambda, \alpha)\neq\emptyset$ if and only if $\alpha=(1,
\alpha_0,\cdots, \alpha_{n-1})$ is a positive root for $Q_{\infty},$ in
which case by \cite[Lemma 2.1]{CEET},
\[
\dim M_{Q_{\infty}}(\lambda, \alpha)=2\alpha_0-\sum (\alpha_i-\alpha_{i+1})^2\geq 0.
\]

It was established in \cite{CEET} that any positive root $\alpha$ for
which $M_{Q_{\infty}}(\lambda, \alpha)\neq\emptyset$ belongs to the $W_{\infty}$-orbit
of $(1, l, \cdots, l),$ where $2l=\dim M_{Q_{\infty}}(\lambda, \alpha).$
Therefore, the roots $\alpha$ for which $M_{Q_{\infty}}(\lambda, \alpha)$ is a
zero-dimensional variety form a single $W_{\infty}$-orbit: that of
$(1,0,\cdots, 0) =: \theta.$ 
The indecomposable projective $\mathcal{O}^{\lambda}$-module
corresponding to $M_{Q_{\infty}}(\lambda, \theta)$ (which is a single point) is
$\mathcal{O}^{\lambda}$, so
$[\mathcal{O}^{\lambda}]=F(M_{Q_{\infty}}(\lambda,\theta)).$
Thus, for some $w\in W_{\infty},$ we have 
\[
[\mathcal{O}^{\lambda}]=F(M_{Q_{\infty}}(w(\lambda), w(\theta)))=F(R_w(M_{Q_{\infty}}(\lambda,\theta))).
\]
Now, recall the commutative diagram \ref{diagram}, which yields that
$[P]=R_w([\mathcal{O}^{w(\lambda)}])$ for some $w\in W.$
So, $\mathcal{O}^{\lambda'}\cong \mathcal{O}^{w(\lambda)}$ as desired.

To show that $\Pic(\mathcal{O}^{\lambda})=\Aut(\mathcal{O}^{\lambda})$, we need to prove that given an invertible $\mathcal{O}^{\lambda}$-bimodule
$M,$ then $M\cong \mathcal{O}^{\lambda}$ as a left module.
Since $\End_{\mathcal{O}^{\lambda}}(M)^{op}\cong \mathcal{O}^{\lambda},$
$[M]$ must correspond to a 0-dimensional quiver variety. By the above
paragraph, we know that all indecomposable projective
$\mathcal{O}^{\lambda}$-modules that correspond to 0-dimensional
varieties $M_{Q_{\infty}}(\lambda, w(\theta))$ must be of the form
$R_{w}([\mathcal{O}^{w(v)}])\cong M$ for some $w\in W_{\infty}.$ So,
$\mathcal{O}^{w(\lambda)}\cong \mathcal{O}^{\lambda}$. Now, by
Bavula--Jordan's isomorphism theorem, we may conclude that $w$ fixes
$\theta$, so $M\cong \mathcal{O}^{\lambda}$ and we are done.
\end{proof}

\begin{concludingremark}
It is natural to expect that Theorem \ref{kleinian} should hold for
generic parameters, and
$\Pic(\mathcal{O}^{\lambda})=\Aut(\mathcal{O}^{\lambda}).$
Here is why the proof of Theorem \ref{main} does not work for non-type
$A$ deformations: automorphism groups of such deformations are very small
(often trivial), so we are no longer able to conclude that if
$\mathcal{O}^{\lambda'}\cong \End_{\mathcal{O}^{\lambda}}(P),$
where $P$ is an indecomposable projective $\mathcal{O}^{\lambda}$-module,
then $P$ must correspond to a zero-dimensional quiver variety via the
bijection between indecomposable projective
$\mathcal{O}^{\lambda}$-modules and quiver varieties. On the other hand,
recall that P.~Levy \cite{L} has described $\Aut(\mathcal{O}^{\lambda})$
for type $D$ Kleinian singularities. The results in \cite{L} combined with
the usual reduction mod $p\gg 0$ technique (should) imply that
$\Pic(\mathcal{O}^{\lambda})=\Aut(\mathcal{O}^{\lambda})$ at least for
generic enough $\lambda.$ We will come back to these questions elsewhere.
\end{concludingremark}

\begin{acknowledgement}
I am grateful to M.~Boyarchenko for telling me about the question by
T.~Hodges, which served as an important motivation for this work.
 I would also like to thank A.~Eshmatov for many useful discussions, in particular for pointing out a gap in the previous version
 of the paper.

\end{acknowledgement}

\end{document}